\documentclass[11pt, reqno,oneside]{amsart}   	
\usepackage{amsthm}
\usepackage{amssymb}
\usepackage{amsmath}
\usepackage{newpxtext,newpxmath}
 \makeatletter
    
    \@addtoreset{equation}{section}
  \makeatother
\newcommand{\lnorm}{\left|\left|}
\newcommand{\rnorm}{\right|\right|}

\newtheorem{thm}{Theorem}[section]
\newtheorem{lem}[thm]{Lemma}

\newtheorem{rem}[thm]{Remark}
\title[Global existence for the Boltzmann equation]{Global existence for the Boltzmann equation in $L^r_v L^\infty_t L^\infty_x$ spaces}
\author{Koya Nishimura}
\begin{document}
\maketitle
\begin{abstract}
We study the Boltzmann equation near a global Maxwellian. We prove the global existence of a unique mild solution with initial data which belong to the $L^r_v L^\infty_x $ spaces where $r \in (1,\infty]$ by using the excess conservation laws and entropy inequality introduced in \cite{r1}.
\end{abstract}
\section{The Boltzmann Equation}
Recall that the Boltzmann equation is given by
\begin{equation}
\label{ws}
\partial_t F+v \cdot \nabla_x F= Q(F,F),\quad F(0,x,v)=F_0 (x,v),
\end{equation}
where $F(t,x,v)$ is the distribution function for the particles at time $t \geq 0$, position $x \in \Omega=\mathbb{R}^3 \text{ or } \mathbb{T}^3$, and velocity $v \in \mathbb{R}^3.$ The collision operator is defined by
\begin{equation}
Q(F,G)(v)=\int_{\mathbb{R}^3 \times \mathbb{S}^2} du\, d\omega\,|v-u|^\gamma b(\theta) \left[F(v')G(u')-F(v)G(u)\right].\nonumber
\end{equation}
Here the angle $\theta$ is defined by $\cos \theta=[v-u]\cdot \omega / |v-u|$ and $B(\theta)$ satisfies the angular cutoff assumption $0 \leq b(\theta)\leq C \left|\cos \theta \right|$. We assume hard potentials $0 \leq \gamma \leq 1$. The post-collisional velocities satisfy
\begin{equation}
\label{cons}
\begin{gathered}
v'=v+[(u-v)\cdot  \omega]\omega,\quad u'=u-[(u-v)\cdot \omega]\omega,\\
v'+u'=v+u,\quad \left|v'\right|^2+\left|u'\right|^2=|v|^2+|u|^2.
\end{gathered}
\end{equation}

Denoting a normalized global Maxwellian by $\mu(v)=e^{-|v|^2}$, $\mu$ satisfies (\ref{ws}) by (\ref{cons}), and so we define the perturbation $f(t,x,v)$ to $\mu$ as $$F=\mu+\sqrt{\mu} f.$$
We consider the Boltzmann equation for the perturbation $f$:
\begin{equation}
\label{lrb}
\left[\partial_t +v \cdot \nabla_x+\nu(v)-K \right]f=\Gamma(f,f),\quad f(0,x,v)= f_0 (x,v).
\end{equation}
Above $\nu (v)=\Gamma_{loss} (1,\sqrt{\mu}) \approx (1+|v|)^\gamma$ is a multiplication operator defined by (\ref{dcc}) below, and $K$ is a integral operator. (the kernel satisfies (\ref{jk1}) below. Also, see \cite{w21} for its form.) Since $Q(\mu,\mu)=0$, the remaining nonlinear part $\Gamma(\cdot,\cdot)$ is defined as $$\Gamma(g,h)={1 \over \sqrt{\mu}}Q(\sqrt{\mu} g, \sqrt{\mu} h).$$
Lastly, the mild form of (\ref{lrb}) is given by
\begin{equation}
\label{integraleq}
\begin{gathered}
\hspace{-50mm}f(t,x,v)=e^{-\nu(v)t}f_0 \left(x-v t,v\right)\\ \hspace{10mm}+\int_0^t e^{-\nu(v)(t-s)} \,\left[K \left(f\right) +\Gamma \left(f,f\right) \right]\left(s,x-v(t-s),v\right)\,ds,
\end{gathered}
\end{equation}
and its equivalent form is given by (\ref{sxa}) below.
\section{Main Results}
\noindent{\bf{Notation.}} In this paper, we use the notation $L^\infty_x =L^\infty (\Omega)$ and $L^r_v=L^r (\Bbb{R}^3_v)\,(r \in (1,\infty])$. We also write the $L^\infty$ norm on the time interval $[0,t]$ as $\lnorm \,\cdot \, \rnorm_{L^\infty_t}$. For a function $g : [0,\infty) \times \Bbb{R}^3_x \times \Bbb{R}^3_v \rightarrow \Bbb{R}$, we define the mixed norms $$\lnorm g\rnorm_{L^\infty_t L^r_v L^\infty_x}=\sup_{s \in [0,t]} \left[\int_{\Bbb{R}^3} du\, \left\{\sup_{y \in \Bbb{R}^3} \left| g(s,y,u) \right| \right\}^r \right]^{1 \over r},$$
and 
$$\lnorm g\rnorm_{L^r_v L^\infty_{t,x}}= \lnorm g\rnorm_{L^r_v L^\infty_t L^\infty_{x}} =\left[ \int_{\Bbb{R}^3} du \,\left\{\sup_{s \in [0,t],\,y \in \Bbb{R}^3} \left|g(s,y,u) \right| \right\}^r \right]^{1 \over r}.$$
Similarly for the norm $|| \cdot ||_{L^r_v L^\infty_x}$. For $r$, we denote the conjugate exponent to $r$ by $r'$. We define a weight function $w (v)=1+|v|$.

For a solution to the Boltzmann equation (\ref{ws}), we have formally the excess conservations of mass and energy and the excess entropy inequality:
\begin{equation}
\label{abba}
 \begin{gathered}
\iint_{\Bbb{R}^3 \times \Bbb{R}^3} F(t,x,v)-\mu(v) \,dv\,dv=\iint F_0-\mu \equiv M_0,\\
\iint_{\Bbb{R}^3 \times \Bbb{R}^3} \left|v \right|^2 \left[F(t,x,v)-\mu(v)\right] \,dv\,dv=\iint \left|v \right|^2 \left[F_0-\mu \right] \equiv E_0,\\
\hspace{-30mm}\iint_{\Bbb{R}^3 \times \Bbb{R}^3} F(t,x,v) \ln F(t,x,v) -\mu(v) \ln \mu(v)\,dv\,dx\\ \hspace{30mm}\leq \iint_{\Bbb{R}^3 \times \Bbb{R}^3} F_0 \ln F_0 -\mu \ln \mu \,dv\,dx \equiv H_0.
 \end{gathered}
\end{equation}
The following local and global existence results are valid.
\begin{thm}
\label{t2}
Let $r \in (1,\infty]$ and $l>\max\{3/r',1/r' + (\gamma+1)/2,2\gamma\}$, and $F_0=\mu+\sqrt{\mu}f_0 \geq 0$. For any $0<M<\infty$, there exist $T^\star (M)>0$ and $\epsilon>0$ such that if $\lnorm w^l f_0 \rnorm_{L^r_v L^\infty_{x}} \leq M/2,$ and 
\begin{equation}
\label{aaff}
\sup_{0 \leq t \leq T^\star,\,x \in \mathbb{R}^3}  \int_{\mathbb{R}^3} dv\,e^{-{|v|^2 \over 4}} \left| f_0 (x-vt,v)\right| \leq \epsilon,
\end{equation}
then there is a unique local solution (\ref{integraleq}), $f(t,x,v)$, to (\ref{lrb}) in $[0,T^\star] \times \Omega \times \Bbb{R}^3$ satisfying
$$\lnorm w^l f \rnorm_{L^r_v L^\infty_{T^\star} L^\infty_x} \leq M,$$
and $F=\mu+\sqrt{\mu} f \geq 0.$ Moreover, if $M_0$, $E_0$, and $H_0$ are finite, then (\ref{abba}) holds.
\end{thm}

\begin{thm}
\label{t1}
In addition to the assumptions as Theorem \ref{t2}, let $l>3/r'+\gamma$. For any $0<M<\infty$, there exist $\epsilon>0$ and $C_0 (r,l)>0$ such that if $|| w^l f_0 ||_{L^r_v L^\infty_{x}} \leq M$, (\ref{aaff}) and 
\begin{equation}
\label{jji}
\sup_{t\geq T^\star,\,x \in \mathbb{R}^3} \int_{\mathbb{R}^3} e^{-\nu(v)t} \left| f_0 (x-vt,v)\right|\,dv+\left|M_0 \right|+\left|E_0 \right|+\left|H_0 \right| \leq \epsilon,
\end{equation}
then there is a unique global solution (\ref{integraleq}), $f(t,x,v)$, to (\ref{lrb}) satisfying
$$\lnorm w^l f \rnorm_{L^r_v L^\infty_{t,x}} \leq C_0 \left(M+M^2\right) \quad \forall t>0,$$
and moreover $F=\mu+\sqrt{\mu}f \geq 0.$
\end{thm}
\begin{rem}
\begin{itemize}
\item[1.] In Theorem \ref{t2}, when $r \in [4/(3-\gamma),\infty]$ we need not assume
\begin{equation}
\label{qll}
\sup_{0 \leq t \leq T^\star,\, x \in \Bbb{R}^3} \int dv\,e^{-{|v|^2 \over 4}}\left|f_0 (x-vt,v)\right| \ll 1,
\end{equation}
and the $L^r_v L^\infty_{t,x}$ norm can be replaced by the $L^\infty_t L^r_v L^\infty_x$ norm in both theorems. (\ref{qll}) is required only when $r \in (1,4/(3-\gamma))$ in Theorem \ref{t2}. In the case, we use Lemma \ref{q11e} to get a decay of the collision term of (\ref{kjb}) as $|v| \rightarrow \infty$. (Note that we consider the hard potential case.)
\item[2.] Recently, the case $r=\infty$ was proved in \cite{r2g}. In our results, we can take large initial data in $L^r_v L^\infty_{x}\,(r>1)$ and we need not take the uniform norm with respect to velocity variable $v$, but the $L^\infty_t$ norm is taken before the $L^r_v$ norm.
\end{itemize}
\end{rem}

In \cite{r1}, the $L^\infty$ estimate using the excess conservation laws and entropy inequality (\ref{abba}) is established. And in \cite{r2g}, to obtain global existence, it was shown that by a similar argument one can make the $L^1_v$ norm of a solution (\ref{integraleq}) small as in Lemma \ref{lem22}. 
Also for the $L^\infty$ estimate of the collision term, the $L^1_v$ norm is involved as in Lemma \ref{ineq1} below. The case $r<\infty$ is technically more complicated to handle than the case $r=\infty$. For instance, in Lemma \ref{ineq1}, we will need to split the integral domain of the gain term into four parts, and we change variables several times. The purpose of this paper is to extend the global existence results of \cite{r2g}. For  historical results of the Boltzmann equation, see the article and the references therein.

This article is organized as follows. In Section 3 we prove local existence (Theorem \ref{t2}). And in Section 4, we establish a $L^\infty_{t,x} L^r_v$ estimate (Lemma \ref{lem22}) and a $L^r_v L^\infty_{t,x}$ estimate (Lemma \ref{lem23}). Global existence (Theorem \ref{t1}) is follows easily from them.
\section{Local Solutions}
As usual, we split  $\Gamma(g,h)=\Gamma_{gain} (g,h)-\Gamma_{loss}(g,h)$ as
\begin{equation}
\begin{split}
\Gamma_{gain} (g,h) (v)\hspace{-0mm}={1 \over \sqrt{\mu (v)}} \int_{\Bbb{R}^3 \times \Bbb{S}^2} du\,d\omega\,|v-u|^\gamma b(\theta) \cdot \left(\sqrt{\mu}g\right)(v') \cdot \left(\sqrt{\mu}h\right)(u'),
\end{split}
\end{equation}
\begin{equation}
\label{dcc}
\Gamma_{loss} (g,h) (v)=\int_{\Bbb{R}^3 \times \Bbb{S}^2} du\,d\omega\,|v-u|^\gamma b(\theta) \sqrt{\mu (u)}g(v)h(u).
\end{equation}
To begin with, we give a estimate for $\Gamma_{gain}$. 
\begin{lem} 
\label{lem1}
Let $r \in[4/(3-\gamma),\infty]$ and $l >3/r'$. For $g(v),h(v) \geq 0$, we have
\begin{equation}
\label{lgain}
\lnorm w^l \Gamma_{gain} (g,h) \rnorm_{L^r_v} \leq C_{r,l} \lnorm w^l g\rnorm_{L^r_v} \lnorm w^l h \rnorm_{L^r_v}.
\end{equation}
\end{lem}
\begin{proof}
We estimate $w^l \Gamma_{gain} (g,h)$ as follows. Since $w^l (v) \leq C w^l (u')+ C w^l (v')$ by (\ref{cons}),
\begin{equation}
\label{gain}
 \begin{split}
&w^l(v) \Gamma_{gain} (g,h) (v)\\ &\hspace{5mm} \leq  {C \over \sqrt{\mu(v)}}\iint |v-u|^\gamma b(\theta) \cdot \left(\sqrt{\mu} g \right) (v') \cdot \left(w^l \sqrt{\mu}h \right)(u')\,du \,d\omega\\ &\hspace{10mm} + {C \over \sqrt{\mu(v)}}\iint |v-u|^\gamma b(\theta) \cdot \left(w^l \sqrt{\mu} g \right) (v') \cdot \left( \sqrt{\mu} h \right)(u')\,du \,d\omega.
\end{split}
\end{equation}
As in Proposition 2.1 of \cite{r2g}, it suffices to estimate only the first term (because, one may interchange $u'$ and $v'$ in the second term. we refer to page 41-42 of \cite{w21}). By H\"older's inequality, the first term is bounded by
\begin{equation}
\label{jhn}
\begin{split}
\iint &|v-u|^\gamma b(\theta) e^{-{|u|^2 \over 2}}w^{-l} (v') \cdot \left(w^l g\right)(v') \cdot \left(w^l h\right)(u') \,du\,d\omega\\&\hspace{5mm} \leq C \left[\iint \left[|v-u|^\gamma \left|\cos \theta \right|\right]^{r'} e^{-{r'}{|u|^2 \over 2}} w^{-{r'l }} (v')\,du\,d\omega \right]^{1 \over r'} \\
& \hspace{30mm} \times \left[ \iint \left|\left(w^l g\right)(v') \cdot \left(w^l h\right)(u') \right|^r\,du\,d\omega\right]^{1 \over r},
 \end{split}
\end{equation}
with the standard modification when $r=\infty$. By changing $u=z+v$ and spliting $z_{||}=[z \cdot \omega]\omega,$ $z_{\perp}=z-z_{||}$, the integral of the first factor can be bounded by
\begin{equation}
\label{qqa}
 \begin{split}
C\iint \left[|z|^{\gamma-1} |z_{||}|\right]^{r'} e^{-{r'}{|z+v|^2 \over 2} } w^{-{r'l}} (z_{||}+v)\left|z_{||}\right|^{-2} dz_{\perp}\,dz_{||}.
 \end{split}
\end{equation}
The further substitution $y=z_{||}+v$ and the inequality $|z| \geq \sqrt{|z_{||}| \cdot |z_{\perp}|}$ yield that (\ref{qqa}) is bounded by
\begin{equation}
\label{ph}
 \begin{split}
\iint &\left|y-v \right|^{{{{\gamma+1} \over 2}{r'}}-2} e^{-{r'}{|y+z_{\perp}|^2 \over 2}}w^{-{r'l}} (y)\,dz_{\perp}\,dy\\ & \leq C \int \left|y-v \right|^{{{\gamma+1} \over 2}{{r'}-2}}w^{-{r'l}} (y)\,dy\\& \leq C_{r,l} (1+|v|)^{{{\gamma+1} \over 2}{r'}-2}.
 \end{split}
\end{equation}
This is bounded if $r \geq 4/(3-\gamma)$, so by taking the $L^r$ norm of (\ref{gain}) and noting that $du\,dv=du'\,dv'$, we can obtain the lemma. 
\end{proof}
\begin{rem}
When $r \in (1,4/(3-\gamma))$ note that (\ref{ph}) is not bounded, but a simple modification of the argument in the proof of the lemma shows that for $l>1/r'+(\gamma+1) /2$, we have
\begin{equation}
\label{bbg}
 \begin{split}
\lnorm w^{l-{{\gamma+1} \over 2}+ {2 \over r'}} \Gamma_{gain} (g,h)\rnorm_{L^r_v} \leq C_r \lnorm w^l g \rnorm_{L^r_v} \lnorm w^l h \rnorm_{L^r_v}.
 \end{split}
\end{equation}  
\end{rem}
We now prove Theorem \ref{t2}. To this end, rewrite the mild form of (\ref{lrb}) as follows.
\begin{equation}
\label{sxa}
 \begin{split}
f(t,x,v)& = \, e^{-\int_0^t g_{f} (s_1,y+vs_1,v)\,ds_1}f_0 (y,v)\\
&\hspace{5mm} + \int_0^t e^{-\int_s^t g_{f} (s_1,y+vs_1,v)\,ds_1} Kf (s,y+vs,v)\,ds\\
& \hspace{5mm}+ \int_0^t e^{-\int_s^t g_{f} (s_1,y+vs_1,v)\,ds_1} \Gamma_{gain}(f,f) (s,y+vs,v)\,ds,
 \end{split}
\end{equation}
Here we have used the notation $y=x-vt$, and for a function $r(s_1,x,v)$,
\begin{equation}
\label{vc}
 \begin{split}
g_{r} (s_1,x,v)=\iint du \,d\omega |v-u|^\gamma b(\theta) \left[\mu(u)+\sqrt{\mu(u)}r (s_1,x,u)\right]. \nonumber
 \end{split}
\end{equation}

\begin{proof}[Proof of Theorem \ref{t2}]
We use the following iterating sequence ($n \geq 0$).
\begin{equation}
\label{kjb}
 \begin{split}
f^{n+1}(t,x,v)& = \, e^{-\int_0^t g_{f^n} (s_1,y+vs_1,v)\,ds_1}f_0 (y,v)\\
& \hspace{5mm}+ \int_0^t e^{-\int_s^t g_{f^n} (s_1,y+vs_1,v)\,ds_1} Kf^n (s,y+vs,v)\,ds\\
& \hspace{5mm}+ \int_0^t e^{-\int_s^t g_{f^n} (s_1,y+vs_1,v)\,ds_1} \Gamma_{gain}(f^n,f^n) (s,y+vs,v)\,ds.
 \end{split}
\end{equation}
We set $f^n (0,x,v)=f_0$ for $n \geq 1$, and $f^0=0.$ 
It is easily verified as in Proposition 2.1 of \cite{r2g} that $\mu+\sqrt{\mu}f^n \geq 0$ if $\mu+\sqrt{\mu}f_0 \geq 0$. First, we consider the case $r \in [4/(3-\gamma),\infty]$ and we will show that if $\sup_{0 \leq t \leq T^\star} || w^l f^n (t) ||_{L^r_v L^\infty_x} \leq M$ then $\sup_{0 \leq t \leq T^\star} || w^l f^{n+1} (t) ||_{L^r_v L^\infty_x} \leq M$. 
We denote by $k(v,u)$ the kernel of $K$ where $k(v,u)$ satisfies
\begin{equation}
\label{jk1}
\left|k (v,u)\right| \leq C |v-u| e^{-{{|v|^2+|u|^2} \over 8}} + C |v-u|^{-1}{e^{-{|v-u|^2 \over 8}-{{\left[|v|^2 -|u|^2\right]^2} \over 8|v-u|^2}}},
\end{equation}
and for $l \in \Bbb{R}$,
\begin{equation}
\label{jk}
\int \left|k(v,u) \right| {{w^l (v)} \over {w^l (u)}}\,du \leq C (1+|v|)^{-1}.
\end{equation}
For the proof, see Lemma 7 of \cite{r112} for instance. (When $l <0$ use the result for $l \geq 0$ and $|u| \leq |u-v|+|v|$ ) From (\ref{jk}) and Lemma \ref{lem1}
we can obtain
\begin{equation}
 \begin{split}
\lnorm w^l f^{n+1} (t) \rnorm_{L^r_v L^\infty_x} \leq \lnorm w^l f^{n+1} (0) \rnorm_{L^r_v L^\infty_x}&+C t \lnorm w^l f^{n} \rnorm_{L^\infty_{t} L^r_v L^\infty_x}\\ & +C_r t  \lnorm w^l f^{n} \rnorm^2_{L^\infty_t L^r_v L^\infty_x}, \nonumber
 \end{split}
\end{equation}
and hence
$$\lnorm w^l f^{n+1} \rnorm_{L^\infty_{T^\star} L^r_v L^\infty_x} \leq M/2+C MT^\star+C_r M^2T^\star \leq M$$
when $T^\star$ is sufficiently small. As for uniqueness, we take the difference $f-h$ where $f$ and $h$ satisfy (\ref{sxa}), as follows.
\begin{equation}
\label{xxc1}
 \begin{split}
w^{{l} \over 2} (v) & \left|\left[f-h\right](t,x,v)\right| \\ & \hspace{-0mm} \leq \left\{w^{{l} \over 2} (v) \left|f_0 (y,v) \right|+\int_0^t w^{{l} \over 2} (v) \left|K f (s,y+vs,v) \right| \,ds\right.\\ & \hspace{0mm}\left.\hspace{25mm}+ \int_0^t w^{{l} \over 2} (v) \left|\Gamma_{gain} (f,f)(s,y+vs,v) \right|\,ds \right\}\\ &\hspace{45mm}
 \times \int_s^t \left|\left[g_f-g_h \right](s_1,y+vs_1,v) \right|\,ds_1\\ & \hspace{5mm}+ \int_0^t w^{{l} \over 2} (v) \left|K\left[f-h \right] (s,y+vs,v)\right|\,ds\\ & \hspace{5mm} +\int_0^t w^{{l} \over 2} (v) \left|\Gamma_{gain} (f-h,f)(s,y+vs,v) \right|\,ds\\ &\hspace{5mm}+\int_0^t w^{{l} \over 2} (v) \left|\Gamma_{gain} (h,f-h)(s,y+vs,v) \right|\,ds.
 \end{split}
\end{equation}
Here we have used the inequality $|e^{-a}-e^{-b}| \leq |a-b|,\,\forall a,b \geq 0$. Note that $l>2 \gamma$ and $\nu(v)=w^\gamma (v).$  Clearly $$\int_s^t \left|\left[g_f-g_h \right](s_1,y+vs_1,v) \right|\,ds_1 \leq C_r w^{l \over 2} (v) t \lnorm f-h \rnorm_{L^\infty_{t}L^r_v L^\infty_x},$$
so the first term on the right hand side of (\ref{xxc1}) is bounded by
$$C_r M t \lnorm  f-h  \rnorm_{L^\infty_t L^r_v L^\infty_x}.$$
By (\ref{cons}), the second term from the last of (\ref{xxc1}) is bounded by
\begin{equation}
 \begin{split}
 C &\iint du\,d\omega\,{w^\gamma (v) \over [w(v')w(u')]^{l \over 2}} e^{-{|u|^2 \over 2}} \left(w^l f \right)(v') \cdot \left(w^{l \over 2} [f-h] \right)(u')\\&+\iint du\,d\omega\,|v-u|^\gamma b(\theta) e^{-{|u|^2 \over 2}}w^{-l} (v') \left(w^l f\right)(v') \cdot \left(w^{l \over 2}[f-h] \right)(u'). \nonumber
 \end{split}
\end{equation}
As in the proof of Lemma \ref{lem1}, its $L^\infty_t L^r_v L^\infty_{x}$ norm is bounded by
$$C_r M t \lnorm w^{l \over 2} [f-h] \rnorm_{L^\infty_{t} L^r_v L^\infty_x}.$$
Similarly for the last term of (\ref{xxc1}). We have
\begin{equation*}
 \begin{split}
\lnorm w^{{l} \over 2} \left[f-h \right]\rnorm_{L^\infty_{T^\star} L^r_v L^\infty_x} \leq C_r MT^\star \lnorm w^{{l} \over 2} \left[ f-h \right]  \rnorm_{L^\infty_{T^\star} L^r_v L^\infty_x}.
 \end{split}
\end{equation*}
Hence uniqueness follows. Similarly, we can also prove that $(f_n)$ is a Cauchy sequence. Letting $n \rightarrow \infty$ we obtain a unique mild solution (\ref{sxa}), $f(t,x,v)$, in $[0,T^\star] \times \Omega \times \Bbb{R}^3$. For the remaining assertions we refer to the proof of Proposition 2.1 of \cite{r2g}. 

For the other case $r \in (1,4/(3-\gamma))$ we replace the $L^\infty_t L^r_v L^\infty_x$ norm by $L^r_v L^\infty_{t,x}$ and use (\ref{bbg}) and the fact $$\int_0^t e^{-\nu(v)(t-s)} \nu^\delta (v)\,ds \leq \eta+C_\eta t,\quad 0 \leq \delta <1,$$
for any $\eta>0$, and Lemma \ref{q11e} below. Taking $\gamma\delta=(\gamma+1)/2-2/r'$, the $L^r_v L^\infty_{t,x}$ norm of the last term of (\ref{kjb}) is bounded by
\begin{equation}
 \begin{split}
 \int_0^t ds \,e^{-\nu(v)(t-s)} \nu^\delta (v)\lnorm w^{l-\gamma \delta} \Gamma_{gain} (f^n,f^n) \rnorm_{L^r_v L^\infty_{t,x}} \leq C_r \eta \lnorm w^l f^n \rnorm^2_{L^r_v L^\infty_{t,x}}. \nonumber
 \end{split}
\end{equation}
 The remaining proof is a simple modification of the case $r \geq 4/(3-\gamma)$.
\end{proof}
With the same assumptions as Theorem \ref{t2}, we have the following lemma for the sequence (\ref{kjb}).
\begin{lem} 
\label{q11e}
For any $\eta>0$, there exists $T^\star (\eta,M)>0$ such that if
$$\sup_{0 \leq t \leq T^\star, \,x \in \Bbb{R}^3}\int dv\,e^{-{|v|^2 \over 4}}\left|f_0 (x-vt,v) \right| \leq \eta/2,$$
and $$\sup_{0 \leq t \leq T^\star}\lnorm w^l f^{n-1} (t) \rnorm_{L^r_v L^\infty_x} \leq M,$$
then 
\begin{equation} 
\label{io}
\sup_{0 \leq s_1 \leq T^\star, \,x\in \Bbb{R}^3}\int du\,e^{-{|u|^2 \over 4}} \left| f^n (s_1,y+vs_1,u) \right| \leq \eta.
\end{equation}
Moreover, we have
\begin{equation}
\label{ioo}
-\int_s^t ds_1\,g_{f^n} (s_1,y+vs_1,v) \leq -{\nu(v)(t-s) /2},
\end{equation}
when $\eta$ is sufficiently small and $0 \leq t \leq T^\star.$
\label{lem2} 
\end{lem}
\begin{proof}
From (\ref{kjb}), we get
\begin{equation*}
\begin{split}
\int & e^{-{|u|^2 \over 4}} \left| f^n (s_1,y+vs_1,u) \right|\,du\\& \leq \int e^{-{|u|^2 \over 4}} \left| f_0 (y+vs_1-us_1,u) \right|\,du\\&\hspace{5mm}+\int_0^{s_1} \iint e^{-{|u|^2 \over 4}}\left| k(u,u_1)\right|\cdot  \lnorm f^{n-1} (s_2,u_1) \rnorm_{L^\infty_x} \,du_1\,du\,ds_2\\&\hspace{5mm}+\int_0^{s_1} \int e^{- {|u|^2 \over 4}} \Gamma_{gain} \left[ \lnorm f^{n-1} \rnorm_{L^\infty_x},\lnorm f^{n-1} \rnorm_{L^\infty_x}\right] (s_2,u)\,du\,ds_2.
\end{split}
\end{equation*}
By (\ref{jk}), the second term on the right hand side is bounded by 
\begin{equation}
 \begin{split}
 \int_0^{s_1} ds_2 & \,\int du \left| k(u,u_1) \right| \int du_1\,\lnorm f^{n-1} (s_2,u_1)\rnorm_{L^\infty_x} \\& \leq C \int_0^{s_1} ds_2 \int du_1\,\lnorm f^{n-1} (s_2,u_1)\rnorm_{L^\infty_x}\\ &\leq C_{r,l} M s_1.\nonumber
 \end{split}
\end{equation}
Using (\ref{cons}) and the fact $du\,dv=du'\,dv'$, the last term is bounded by
\begin{equation}
 \begin{split}
C \int_0^{s_1} ds_2 \iiint & du\,du_1\,d\omega \, e^{-{1 \over 8} \left[|u'|^2+|u'_1|^2 \right]} \lnorm f^{n-1} (s_2,u')\rnorm_{L^\infty_x} \cdot \lnorm f^{n-1} (s_2,u'_1)\rnorm_{L^\infty_x} \\& \leq C \int_0^{s_1} ds_2 \left[\int du'\, \lnorm f^{n-1} (s_2,u')\rnorm_{L^\infty_x} \right]^2  \\ & \leq C_{r,l} M^2 s_1.\nonumber
 \end{split}
\end{equation}
We have
\begin{equation*}
\begin{split}
\int e^{-{|u|^2 \over 4}} \left| f^n (s_1,y+vs_1,u) \right|\,du &\leq \eta/2+C_{r,l} Ms_1+C_{r,l} M^2 s_1,
\end{split}
\end{equation*}
and hence we have (\ref{io}) if we choose $T^\star$ small. Moreover, (\ref{ioo}) follows from this and
\begin{equation*}
\begin{split}
-g_{f^n} &(s_1,y+vs_1,v) \\&\leq -\nu(v)+\iint |v-u|^\gamma b(\theta)\sqrt{\mu(u)} \left| f^n (s_1,y+vs_1,u) \right|\,du\,d\omega\\& \leq -\nu(v)+C\nu(v) \int e^{-{|u|^2 \over 4}} \left| f^n (s_1,y+vs_1,u) \right|\,du.
\end{split}
\end{equation*}
\end{proof}
\section{Global Existence}
It is important to bound the nonlinear term by using the $L^1_v$ norm.
\begin{lem} \label{ineq1}
Let $r \in (1,\infty]$, $l > 3/r'$, $n>3$, and $g(v) \geq 0.$ For any $\eta>0$ we have
\begin{equation}
\label{sscv}
 \begin{split}
\lnorm w^{l-\gamma} \Gamma_{gain} (g,g) \rnorm_{L^r_v} \hspace{0mm}\leq C_{r,\eta} \lnorm g \rnorm^{1 \over nr'}_{L^1_v} \lnorm w^l g \rnorm^{{1}+{1 \over r}+{1 \over n'r'}}_{L^r_v} + C_{r,l} \eta \sum_{p=1,r} \lnorm w^l g \rnorm^{1+{1 \over p}}_{L^r_v},
 \end{split}
\end{equation}
and 
\begin{equation}
\lnorm w^{l-\gamma} \Gamma_{loss} (g,g) \rnorm_{L^r_v} \leq C_r \lnorm g \rnorm_{L^1_v} \lnorm w^l g \rnorm_{L^r_v}.
\end{equation}
\end{lem}
\begin{proof} The inequality for the loss term is trivial, so we only estimate the gain term. Using (\ref{cons}) and then interchanging $v'$ and $u'$ as in Lemma \ref{lem1},
\begin{equation}
\label{ineq2}
 \begin{split}
w^{l-\gamma}(v)\Gamma_{gain}(g,g)(v)& \leq C \iint \left[w^l(v')+w^l (u')\right] e^{-{{|u|^2} \over 2}} g(v')g(u')\,du\,d\omega \\ & \leq C\iint g(v') \cdot e^{-{{|u|^2} \over 2}} \cdot \left(w^l g\right)(u')\,du\,d\omega. \nonumber
 \end{split}
\end{equation}
We split the last integral into four parts. First, for $L>0$,
\begin{equation}
\label{aac}
 \begin{split}
C &\iint_{|u| \geq L} g(v') \cdot e^{-{{|u|^2} \over 2}} \cdot \left(w^l g\right)(u') \,du \,d\omega \\ & \leq C e^{-{L^2 \over 4}} \iint e^{-{{|u|^2} \over 4}} \cdot g (v') \cdot \left(w^l g\right)(u') \,du \,d\omega \\ & \leq C e^{-{L^2 \over 4}} \left[ \iint \left(w^l g \right)^r (v') \cdot \left(w^l g\right)^r (u') \,du \,d\omega \right]^{1 \over r},
 \end{split}
\end{equation}
so by $du\,dv =du'\,dv'$, the $L^r_v$ norm of (\ref{aac}) is bounded by $C e^{-{L^2 \over 4}} || w^l g ||^2_{L^r_v}$. Next, let ${|v| \leq 2L}$ and set $k=1+(r-1)/n'$ for fixed $n>3$. Then we get $1=1/(nr')+k/r$ and
\begin{equation}
\label{km}
 \begin{split}
 C&\iint_{|u| \leq L} g(v') \cdot e^{-{{|u|^2} \over 2}} \cdot \left(w^l g\right)(u')\,du\,d\omega \\& \leq C \left[ \iint_{|u| \leq L} g^{1 \over n} (v') \,e^{-{{|u|^2} \over 2}} \,du\,d\omega \right]^{1 \over r'} \left[\iint_{|u| \leq L} g^{k} (v') \cdot \left(w^l g \right)^{r} (u')\,du\,d\omega \right]^{1 \over r}.
 \end{split}
\end{equation}
For the first factor, we use the same change of variables as (\ref{qqa}) and (\ref{ph}). The following integral calculus holds. 
\begin{equation}
\label{po}
 \begin{split}
 \iint_{|u| \leq L} g^{1 \over n} (v') e^{-{|u|^2 \over 2}}\,du\,d\omega & \leq C \iint g^{1 \over n} (v+z_{||}) e^{-{{|v+z|^2} \over 2}} \, {1 \over {|z_{||}|^2}}\,dz_{||}\,dz_{\perp} \\ & \leq C \int_{|y| \leq 5L} g^{1 \over n} (y) {1 \over {|y-v|^{2}}}\,dy \\ & \leq C \left[ \int_{|y| \leq 5L} g (y)\,dy \right]^{1 \over n} \left[\int_{|y| \leq 5L} {1 \over {|y-v|^{2n'}}}\,dy \right]^{1 \over n'} \\ &\leq C_{r,L} \lnorm g \rnorm^{1 \over n}_{L^1_v}.
 \end{split}
\end{equation}
By $du\,dv=du'\,dv'$ and $k<r$, the $L^r (\{|v| \leq 2L\})$ norm of the second factor of (\ref{km}) is bounded by
\begin{equation}
 \begin{split}
 C_{r,L} & \left[\int_{|v'| \leq 5L} dv' g^{k} (v') \int_{|u'| \leq 5L} du'\, \left(w^l g \right)^{r} (u') \right]^{1 \over r} \\ & \leq C_{r,L} \lnorm g \rnorm^{k \over r}_{L^r_v} \lnorm w^l g \rnorm_{L^r_v}\\& \leq C_{r,L} \lnorm w^l g \rnorm^{{k \over r}+1}_{L^r_v}.
 \end{split}
\end{equation}
The $L^r (\{|v| \leq 2L\})$ norm of (\ref{km}) is bounded by $C_{r,L} ||g||^{1 \over nr'}_{L^1_v} ||w^l g||^{1+{k \over r}}_{L^r_v}$. Lastly, when $|v| \geq 2L$, we consider the two cases $|(u-v) \cdot \omega| \leq L$ and $|(u-v) \cdot \omega|\geq L$. For the former, then $|v'| = |v+[(u-v) \cdot \omega] \omega| \geq 2L-L=L$, so
\begin{equation}
\label{km1}
 \begin{split}
 C & \iint_{|(u-v) \cdot \omega| \leq L} e^{-{|u|^2 \over 2}} \cdot g(v') \cdot \left(w^l g \right)(u')\,du\,d\omega \\& \leq C L^{-l} \iint  e^{-{|u|^2 \over 2}} \cdot \left(w^l g \right) (v') \cdot \left(w^l g \right) (u')\,du\,d\omega.
 \end{split}
\end{equation}
As in (\ref{aac}), the $L^r (\{|v| \geq 2L\})$ norm of (\ref{km1}) is bounded by $CL^{-l} ||w^l g||^2_{L^r_v}$. In the latter, since $|z_{||}| \geq L$, as in (\ref{km}) and (\ref{po}) (take $n=1$), we can get 
\begin{equation}
 \begin{split}
 \lnorm C  \iint_{|(u-v) \cdot \omega| \geq L} e^{-{|u|^2 \over 2}} \cdot g(v') \cdot \left(w^l g \right)(u')\,du\,d\omega \rnorm_{L^r (\{|v| \geq 2L\})} \leq C_{r,l} L^{-{2 \over r'}} \lnorm w^l g \rnorm^{1+{1 \over r}}_{L^r_v},\nonumber
\end{split}
\end{equation}
and then we have (\ref{sscv}) by collecting above estimates and choosing $L$ large.
\end{proof}
For simplicity, we use the notation $\mathcal{E}_0 =\left[\left|M_0\right|+\left|E_0\right|+\left|H_0\right|\right]^{m^{-1}}$ for $0<m<1$ sufficiently small, and $k_l (v,u)=k(v,u) \nu^l (v) / \nu^l (u)$ where $k(v,u)$ is the kernel of the integral operator $K$. As in \cite{r1} or \cite{r2g}, when $t-s \geq \kappa \,(0<\kappa<1),\,N>0$, we can obtain
\begin{equation}
\label{bbf}
 \begin{split}
\iint_{|v| \leq 5N,\, |u| \leq 5N} \left|f(s,x-v(t-s),u)\right|\,du\,dv \leq C_N \left(1+{\kappa^{-1}} \right) \mathcal{E}_0,
 \end{split}
\end{equation}
which is the key estimate to global solvability. Recall that $n$ and $k$ were defined in the proof of Lemma \ref{ineq1}. Under the assumption of Theorem \ref{t1}, from (\ref{bbf}), the following two lemmas are valid and then Theorem \ref{t1} follows easily (see Proof of Theorem 1.1 in \cite{r2g}).
\begin{lem}
\label{lem22}
For any $\eta>0$ there exists $C_\eta (r,l)>0$ such that
\begin{equation}
 \begin{split}
\sup_{s \in [T^\star,t]} \lnorm f (s) \rnorm_{L^\infty_x L^1_v} \leq \epsilon+C_{r,l} \eta \sum_{p_1 =1,2} & \lnorm w^l f \rnorm^{p_1}_{L^r_v L^\infty_{t,x}}+C_\eta \mathcal{E}_0\\&+C_{\eta} \mathcal{E}_0 \sum_{p_2 =r,r/k} \lnorm w^l f \rnorm^{1+{1 \over p_2}}_{L^r_v L^\infty_{t,x}},
 \end{split}
\end{equation}
\end{lem}
\begin{rem} 
On the interval $[0,T^\star]$, Theorem \ref{t2} yields $||f||_{L^\infty_{T^\star} L^\infty_x L^1_v} \leq C_{r,l} M.$
\end{rem}
\begin{lem} For any $\eta>0$ there exists $C_\eta (r,l)>0$ such that
\label{lem23}
\begin{equation}
\label{xxc}
 \begin{split}
 \lnorm w^l f \rnorm_{L^r_v L^\infty_{t,x}}\leq &CM+C_{\eta} \mathcal{E}_0 +C_{r} \lnorm f \rnorm_{L^\infty_{t,x} L^1_v} \lnorm w^l f \rnorm_{L^r_v L^\infty_{t,x}} \\&+ C_{\eta} \lnorm f \rnorm^{1 \over nr'}_{L^\infty_{t,x} L^1_v} \lnorm w^l f \rnorm^{{1}+{k \over r}}_{L^\infty_{t,x} L^r_v}  \hspace{0mm}+ C_{r,l} \eta \sum_{p=\infty,1,r} \lnorm w^l f \rnorm^{1+{1 \over p}}_{L^r_v L^\infty_{t,x}}.
 \end{split}
\end{equation}
\end{lem}
\begin{proof}[Proof of Lemma \ref{lem22}] From (\ref{integraleq}),
\begin{equation}
\begin{split}
 \int \left| f(t,x,v) \right|\,dv \leq \sum_{j=1}^4 G_j (t,x),\nonumber
\end{split}
\end{equation}
where
\begin{equation}
 \begin{split}
G_1 (t,x)=\int e^{-\nu(v)t} \left|f(y_1,v) \right|\,dv,\nonumber
 \end{split}
\end{equation}
\begin{equation}
 \begin{split}
G_2 (t,x)=\int_0^t ds \int dv \int du \, e^{-\nu(v)(t-s)} \left| k (v,u) f (s,y_1+vs,u) \right|,\nonumber
 \end{split}
\end{equation}
\begin{equation}
 \begin{split}
 G_3 (t,x)= \int_0^t ds \int dv \iint & du\,d\omega \, e^{-\nu(v)(t-s)} |v-u|^\gamma b(\theta) \sqrt{\mu(u)}\\ & \times \left| f (s,y_1+vs,u) f (s,y_1+vs,v) \right|,\nonumber
 \end{split}
\end{equation}
\begin{equation}
 \begin{split}
 G_4 (t,x)=\int_0^t ds \int dv \iint & du\,d\omega \, e^{-\nu(v)(t-s)} \left|v'-u' \right|^\gamma b(\theta) \sqrt{\mu(u)}\\ & \times \left| f (s,y_1+vs,u') f (s,y_1+vs,v') \right|.\nonumber
 \end{split}
\end{equation}
Here we have used the notation $y_1=x-vt$. Note that $|v-u|=|v'-u'|$. We further split $G_j,\, j=2,3,4$ as
\begin{equation}
 \begin{split}
 G_j (t,x) & =\int_{t-\kappa}^t \iint  + \int_0^{t-\kappa}  \int \int_{|u| \geq N}+ \int_0^{t-\kappa}  \int_{|v| \geq 2N} \int_{|u| \leq N} \\& \hspace{19.5mm} +\int_0^{t-\kappa}  \int_{|v| \leq 2N} \int_{|u| \leq N} \{ \cdots \}\, du\,dv\,ds \\ &\equiv G_{j1} (t,x)+G_{j2} (t,x)+G_{j3} (t,x)+G_{j4} (t,x).\nonumber
 \end{split}
\end{equation}
By assumption $G_1 \leq \epsilon.$ First we will show that for any $\eta>0$, if $\kappa$ and $N^{-1}$ are sufficiently small, then $G_{jk} \leq \eta \sum_{p_1 =1,2} || w^l f ||^{p_1}_{L^r_v L^\infty_{t,x}}$ for $j=2,3,4,\,k=1,2,3.$ To see this for $G_{2k}$, recall (\ref{jk}). For $G_{2k},\,k=1,2,3$ we integrate over $v$ before $u$. It is not hard to see that $G_{21} \leq C_{r,l} \kappa || w^l f ||_{L^r_v L^\infty_{t,x}}$. Also,
\begin{equation}
 \begin{split}
 G_{22} & \leq \int_0^{t-\kappa} ds\,e^{-\nu(v)(t-s)} \int_{|u| \geq N} du\, \left(\int dv\,\left|k(v,u) \right| \right) \lnorm f(u) \rnorm_{L^\infty_{t,x}} \\ & \leq CN^{-1} \int du\,\lnorm  f(u) \rnorm_{L^\infty_{t,x}} \leq C_{r,l} N^{-1} \lnorm w^l f \rnorm_{L^r_v L^\infty_{t,x}}.\nonumber
 \end{split}
\end{equation}
Let $|v| \geq 2N$ and $|u| \leq N$. Then $|v-u| \geq N$. Since there is also the case $\gamma=0$, we estimate $G_{23}$ as follows. 
\begin{equation}
 \begin{split}
 G_{23} & \leq \int_{|u| \leq N} du \left(\int_{|v| \geq 2N} dv\,\left|k (v,u) \right|\right) \left(\int_0^{t-\kappa} ds\,e^{-\nu(v)(t-s)}\right)\lnorm f(u) \rnorm_{L^\infty_{t,x}} \\ & \leq C \int du \left(\int dv\,e^{- {N^2 \over 16}} \left|k (v,u) \right| e^{{|v-u|^2 \over 16}} \right) \lnorm f(u) \rnorm_{L^\infty_{t,x}}
 \\& \leq C e^{- {N^2 \over 16}}  \int du\,\lnorm f (u) \rnorm_{L^r_v L^\infty_{t,x}} \\ & \leq C_{r,l} e^{- {N^2 \over 16}}  \lnorm w^l f \rnorm_{L^r_v L^\infty_{t,x}},\nonumber
 \end{split}
\end{equation}
Thus the claim for $G_{2k}$ follows by choosing $\kappa$ and $N^{-1}$ small. The terms $G_{3k}$ and $G_{4k}$ are easy to estimate. Noting $du\,dv=du'\,dv'$ and $3-(l-\gamma)r'<0$, we have 
\begin{equation}
 \begin{split}
 \sum_{k=1,2,3} G_{3k}+ G_{4k} & \leq \int_{t-\kappa}^t ds\, e^{-\nu(v)(t-s)} \left[\int dv\, w^{\gamma} (v)\lnorm f (v) \rnorm_{L^\infty_{t,x}} \right]^2 \\ & \hspace{5mm}+ \int_0^{t-\kappa} ds\, e^{-\nu(v)(t-s)} \left[\int_{|v| \geq N} dv\, w^{\gamma} (v)\lnorm f (v) \rnorm_{L^\infty_{t,x}} \right]^2 \\& \leq C_{r,l} \kappa \lnorm w^l f \rnorm^2_{L^r_v L^\infty_{t,x}} + CN^{3-(l-\gamma)r'}  \lnorm w^l f \rnorm^2_{L^r_v L^\infty_{t,x}}.\nonumber
 \end{split}
\end{equation}
Next, we estimate $G_{j4},\,j=2,3,4$ by using (\ref{bbf}) as follows. As in (\ref{po}), set $k=1+(r-1)/n'$ for fixed $n>3$. From the same calculus as (\ref{po}), we have
\begin{equation}
\label{llk}
 \begin{split}
\int_{|v|\leq 2N} & dv \int_{|u| \leq N} du \int d\omega \,\left|f(s,y_1+vs,v') f (s,y_1+vs,u') \right|
\\ & \leq \left[\iiint_{|v|\leq 2N,\,|u| \leq N} \left|f(s,y_1+vs,v')\right|^{1 \over n}  \, du\, dv\,d\omega \right]^{1 \over r'} \\ & \hspace{15mm} \times \left[\int_{|v'| \leq 5N} \lnorm f(v')\rnorm^k_{L^\infty_{t,x}} dv' \int_{|u'| \leq 5N} \lnorm f (u') \rnorm^{r}_{L^\infty_{t,x}} du' \right]^{{1} \over r} 
\\ & \leq C_{r,N} \left[\iint_{|v|\leq 2N,\, |y| \leq 5N} \left|f(s,y_1+vs,y)\right| \,dy\,dv \right]^{1 \over nr'} \\ & \hspace{15mm} \times \left[ \iint_{|v|\leq 2N,\, |y| \leq 5N} {1 \over {|y-v|^{2n'}} } \,dy\,dv\right]^{{1 \over n'}}  \lnorm w^l f \rnorm^{1+{k \over r}}_{L^r_v L^\infty_{t,x}}
\\ & \leq C_{r,N,\kappa} \mathcal{E}_0 \lnorm w^l f \rnorm^{1+{k \over r}}_{L^r_v L^\infty_{t,x}}.
 \end{split}
\end{equation}
Moreover, we easily get  
\begin{equation}
\label{llk2}
 \begin{split}
\int_{|v|\leq 2N} & dv \int_{|u| \leq N} du\,\left|f(s,y_1+vs,v) f (s,y_1+vs,u) \right|\\ & \leq  \left[\iint_{|v|\leq 2N,\,|u| \leq N} \left|f(s,y_1+vs,u)\right| du\,dv \right]^{1 \over r'} \\ & \hspace{5mm} \times \left[\iint_{|v|\leq 2N,\,|u| \leq N} \left|f(s,y_1+vs,u)\right| \cdot \lnorm f (v) \rnorm^r_{L^\infty_{t,x}} du\,dv \right]^{1 \over r} \\ & \leq \left[\iint_{|v|\leq 2N,\,|u| \leq N} \left|f(s,y_1+vs,u)\right| du dv \right]^{1 \over r'} \\ & \hspace{25mm} \times \left[\int \lnorm f(u)\rnorm_{L^\infty_{t,x}} du \int \lnorm f (v) \rnorm^r_{L^\infty_{t,x}} dv \right]^{1 \over r}\\& \leq C_{r,N,\kappa} \mathcal{E}_0 \lnorm w^l f \rnorm^{1+{1 \over r}}_{L^r_v L^\infty_{t,x}}.
 \end{split}
\end{equation}
Hence
\begin{equation}
 \begin{split}
 G_{34}+G_{44} & \leq C_{r,N,\kappa} \mathcal{E}_0 \lnorm w^l f \rnorm^{1+{1 \over r}}_{L^r_v L^\infty_{t,x}} \int_0^{t-\kappa} ds \, e^{-\nu(v)(t-s)} \\ & \leq C_{r,N,\kappa} \mathcal{E}_0 \sum_{p_2 =r,r/k} \lnorm w^l f \rnorm^{1+{1 \over p_2}}_{L^r_v L^\infty_{t,x}}.\nonumber
 \end{split}
\end{equation}
For $G_{24}$, in view of (\ref{jk1}), we need to approximate $k_l (v,u)$ by $k_{l,N} (v,u)$ smooth with compact support such that
\begin{equation}
\label{hjk1}
 \begin{split}
\sup_{|u| \leq 3N} \int_{|v| \leq 3N} dv\,\left| k_l (v,u)-k_{l,N} (v,u) \right| \leq N^{-{4 \over r'}}.
 \end{split}
\end{equation}
We have
\begin{equation}
 \begin{split}
 \int_{|v| \leq 2N} & dv  \int_{|u| \leq N} du \left| k_{l,N} (v,u) \cdot \left(\nu^l f \right) (s,y+vs,u)\right| \\ & \leq C_{l,N} \int_{|v| \leq 2N} dv  \int_{|u| \leq N} du\, \left| f (s,y+vs,u)\right| \\ & \leq C_{l,N,\kappa} \mathcal{E}_0,\nonumber
 \end{split}
\end{equation}
\begin{equation}
 \begin{split}
\int_{|v| \leq 2N,\,|u| \leq N} & du\,dv\, \left| k_l (v,u)-k_{l,N} (v,u) \right| \cdot \left| \left(w^l f \right) (s,y+vs,u)\right| \\ & \leq {C_r N^{-{1 \over r'}}} \lnorm w^l f \rnorm_{L^r_v L^\infty_{t,x}}.\nonumber
 \end{split}
\end{equation}
Hence 
\begin{equation}
 \begin{split}
G_{24} \leq C_{l,N,\kappa} \mathcal{E}_0+{C_r N^{-{1 \over r'}}} \lnorm w^l f \rnorm_{L^r_v L^\infty_{t,x}}.\nonumber
 \end{split}
\end{equation}
We obtain the lemma by collecting above estimates and choosing $\kappa$ and $N^{-1}$ small.
\end{proof}

\begin{proof}[Proof of Lemma \ref{lem23}] From now on we use the notation
\begin{equation}
 \begin{split}
 y_1=x-v(t-s_1), \quad  y_2=y_1-u_1 (s_1-s_2),\nonumber
 \end{split}
\end{equation}
and recall that $k_l (v,u)=w^l (v) k(v,u)/w^l (u)$. By applying (\ref{integraleq}) to the second term on the right hand side of (\ref{integraleq}), we have
\begin{equation}
 \begin{split}
w^l (v) f (t,x,v)=\sum_{j=1}^5 H_j (t,x,v),\nonumber
 \end{split}
\end{equation}
where
\begin{equation}
 \begin{split}
H_1 (t,x,v)=e^{-\nu(v)t} w^l (v) f_{0} (x-vt,v),\nonumber
 \end{split}
\end{equation}
\begin{equation}
 \begin{split}
H_2 &(t,x,v)=\int du_1 \,k_l (v,u_1) \int_0^t ds_1\, e^{-\nu(v)(t-s_1)} \\ & \hspace{25mm} \times e^{-\nu(u_1)s_1} w^l (u_1) f_{0} (y_1-u_1s_1 , u_1),\nonumber
 \end{split}
\end{equation}
\begin{equation}
 \begin{split}
 H_3 &(t,x,v)=\iint du_1\,du_2\, k_l (v,u_1) k_l (u_1,u_2) \int_0^t ds_1\, e^{-\nu(v)(t-s_1)}\\ &\hspace{25mm}\times \int_0^{s_1} ds_2\, e^{-\nu(u_1)(s_1-s_2)} w^l (u_2) f(s_2,y_2,u_2),\nonumber
 \end{split}
\end{equation}
\begin{equation}
 \begin{split}
  H_4 &(t,x,v)=\int du_1\, k_l (v,u_1) \int_0^t ds_1\, e^{-\nu(v)(t-s_1)}\\ &\hspace{25mm}\times \int_0^{s_1} ds_2\, e^{-\nu(u_1)(s_1-s_2)} \,w^l (u_1)\Gamma [f,f](s_2,y_2,u_1),\nonumber
 \end{split}
\end{equation}
\begin{equation}
 \begin{split}
 H_5 (t,x,v)=\int_0^t ds_1 \,e^{-\nu(v)(t-s_1)} \,w^l (v) \Gamma [f,f] \left(s_1,x-v(t-s_1),v\right).\nonumber
 \end{split}
\end{equation}
Clearly $||H_1 ||_{L^r_v L^\infty_{t,x}}$ and $||H_2 ||_{L^r_v L^\infty_{t,x}}$ are bounded by $CM$. For $H_4$ and $H_5$ we can apply Lemma \ref{ineq1}, so their $L^r_v L^\infty_{t,x}$ norm are bounded by the last three terms of (\ref{xxc}). Thus it remains only to estimate $H_3$. We compute the $L^r_v L^\infty_{t,x}$ norm of $H_3$ by dividing it into four parts. First, by repeating H\"older's inequality, we can get
\begin{equation}
\label{wwf}
 \begin{split}
  \iint &du_1\,du_2\, k(v,u_1) k(u_1,u_2) w^l (v) f(s_2,y_2,u_2)\\& \leq C \left[\int du_1\,\left| k_l (v,u_1) \right| \right]^{1 \over r'}\\& \hspace{10mm} \times \left[ \iint du_1\,du_2\, \left|k_l (v,u_1) k_l (u_1,u_2) \right| \lnorm w^l f (u_2) \rnorm^r_{L^\infty_{t,x}} \right]^{1 \over r},
 \end{split}
\end{equation}
so the $L^r (\{|v| \geq L\})$ norm of $H_3$ is bounded by $C_r N^{-{{1} \over r'}} || w^l f ||_{L^r_v L^\infty_{t,x}}$. Since the first factor of (\ref{wwf}) is bounded, we also have
\begin{equation}
\label{sswq}
 \begin{split}
 \lnorm H_3 \rnorm_{L^r_v L^\infty_{t,x}} \leq C \left[\int dv \,\left| k_l (v,u_1) \right| \int du_1 \,\left|k_l (u_1,u_2) \right| \int du_2 \, \lnorm w^l f (u_2) \rnorm^r_{L^\infty_{t,x}}\right]^{1 \over r},
 \end{split}
\end{equation}
and hence if either $|v| \leq N,\,|u_1| \geq 2N$ or $|u_1| \leq 2N,\,|u_2| \geq 3N$ then $|| H_3 ||_{L^r_v L^\infty_{t,x}} \leq C_r N^{-{1 \over r}} || w^l f ||_{L^r_v L^\infty_{t,x}}.$ Lastly, we consider the $L^r (\{|v| \leq N\})$ norm of the remaining part of $H_3$ which is given by
\begin{equation}
\label{lkd}
 \begin{split}
 H_3 (t,x,v)&=\int_{|u_1| \leq 2N} du_1 \int_{|u_2| \leq 3N} du_2\, k_l (v,u_1) k_l (u_1,u_2) \int_\kappa^t ds_1\, e^{-\nu(v)(t-s_1)}\\ &\hspace{25mm}\times \int_{s_1 -\kappa}^{s_1} ds_2\, e^{-\nu(u_1)(s_1-s_2)} w^l (u_2) f(s_2,y_2,u_2)\\& \hspace{5mm}+\int_{|u_1| \leq 2N} du_1 \int_{|u_2| \leq 3N} du_2\, k_l (v,u_1) k_l (u_1,u_2) \int_0^\kappa ds_1\, e^{-\nu(v)(t-s_1)}\\ &\hspace{25mm}\times \int_0^{s_1} ds_2\, e^{-\nu(u_1)(s_1-s_2)} w^l (u_2) f(s_2,y_2,u_2)\\ & \hspace{5mm} +\int_{|u_1| \leq 2N} du_1 \int_{|u_2| \leq 3N} du_2\, k_l (v,u_1) k_l (u_1,u_2) \int_\kappa^t ds_1\, e^{-\nu(v)(t-s_1)}\\ &\hspace{25mm}\times \int_0^{s_1 -\kappa} ds_2\, e^{-\nu(u_1)(s_1-s_2)} w^l (u_2) f(s_2,y_2,u_2).
 \end{split}
\end{equation}
From (\ref{sswq}), clearly, the $L^r (\{|v| \leq N\}; L^\infty_{t,x})$ norms of the first two terms are bounded by $C_{r,N} \kappa || w^l f ||_{L^r_v L^\infty_{t,x}}$. For the last term, we use (\ref{bbf}). As before, we approximate $k_l$ by $k_{l,N}$ satisfying (\ref{hjk1}). Then
\begin{equation}
 \begin{split}
 k_l (v,u_1)  k_l (u_1,u_2)=&\left[k_{l} (v,u_1)-k_{l,N}(v,u_1)\right] k_l (u_1,u_2)\\&+ \left[k_l (u_1,u_2)-k_{l,N} (u_1,u_2) \right] k_{l,N} (v,u_1)\\&+k_{l,N} (v,u_1) k_{l,N} (u_1,u_2),\nonumber
 \end{split}
\end{equation}
and from this and (\ref{sswq}), the $L^r (\{|v| \leq N\} ; L^\infty_{t,x})$ norm of the last term of (\ref{lkd}) is bounded by
\begin{equation}
 \begin{split}
 C_r  N^{-{4 \over rr'}} & \lnorm w^l f \rnorm_{L^r_v L^\infty_{t,x}}\\+&C_{l,N} \lnorm \int_\kappa^t ds_1\, e^{-\nu(v)(t-s_1)} \int_0^{s_1 -\kappa} ds_2\, e^{-\nu(u_1)(s_1-s_2)} \right. \right. \\& \hspace{20mm} \times \left.\left.  \int_{|u_1| \leq 2N} du_1 \int_{|u_2| \leq 3N} du_2 \,\left|f(s_2,y_2,u_2)\right| \rnorm_{L^r_v L^\infty_{t,x}}\\& \hspace{-5mm}\leq {C_r N^{-{4 \over rr'}}} \lnorm w^l f \rnorm_{L^r_v L^\infty_{t,x}}+C_{l,N,\kappa} \mathcal{E}_0.\nonumber
 \end{split}
\end{equation}
Hence we can obtain the lemma from above estimates.
\end{proof}

{\it{E-mail address}}: nishimura.koya.42e@kyoto-u.jp  
\end{document}